\documentclass{amsart}
\usepackage{graphicx}
\usepackage{amsmath}
\usepackage{amsfonts}
\usepackage{amssymb}
\usepackage{ifpdf}
\newtheorem{theorem}{Theorem}

\newtheorem{corollary}[theorem]{Corollary}

\newtheorem{lemma}[theorem]{Lemma}

\newtheorem{remark}[theorem]{Remark}

\begin{document}

\title[]{Sobolev embeddings for Fractional Haj{\l}asz-Sobolev spaces in the
setting of rearrangement invariant spaces}
\author{Joaquim Mart\'{i}n$^{\ast}$}
\address{Department of Mathematics\\
Universitat Aut\`{o}noma de Barcelona} \email{jmartin@mat.uab.cat}
\author{Walter A.  Ortiz**}
\address{Department of Mathematics\\
Universitat Aut\`{o}noma de Barcelona} \email{waortiz@mat.uab.cat}
\thanks{$^{\ast}$Partially supported by Grants  MTM2016-75196-P, MTM2016-77635-P (MINECO) and 2017SGR358, 2017SGR395 (AGAUR, Generalitat de Catalunya)}
\thanks{**Partially supported by Grant 2017SGR395 (AGAUR, Generalitat de Catalunya)}
\thanks{This paper is in final form and no version of it will be submitted for
publication elsewhere.} 
\thanks{Conflict of Interest. The authors declare that they have no conflict
of interest.}

\begin{abstract}
We obtain symmetrization inequalities in the context of Fractional
Haj{\l}asz-Sobolev spaces in the setting of rearrangement invariant
spaces and prove that for a large class of measures our
symmetrization inequalities are equivalent to the lower bound of the
measure.
\end{abstract}
\maketitle

\section{Introduction}

Let us consider a metric measure space $\left(\Omega,d,\mu\right)$
where $\mu$ is a Borel measure on $(\Omega,d)$ such
$0<\mu(B)<\infty$, for every ball $B$ in $\Omega.$ We will always
assume $\mu(\Omega)=\infty$ and $\mu(\left\{ x\right\}  )=0$ for all
$x\in\Omega.$ Let $X$ be a rearrangement invariant (r.i.) space on
$\Omega$ (see section \ref{rear} below). In this paper, we introduce
the fractional Haj{\l}asz-Sobolev spaces $M^{s,X}\left(
\Omega\right) $ for $s>0,$ and we will focus on understanding the
relation between Sobolev embeddings theorems for spaces
$M^{s,X}\left(  \Omega\right) $ and the growth properties of the
measure $\mu$.

Let $s>0$ and let $X$ be a r.i. space on $\Omega.$ We say that $f\in
M^{s,X}\left(  \Omega\right)  $, if $f\in X,$ and there exits a
non-negative measurable function $g\in X$ such that%

\begin{equation}
\left|  f(x)-f(y)\right|  \leq d(x,y)^{s}\left(  g(x)+g(y)\right)  \text{
\ \ }\mu-a.e.\text{ \ }x,y\in\Omega.\label{sgradi}%
\end{equation}

A function $g$ satisfying (\ref{sgradi}) will be called a
$s-$gradient of $f$. We denote by $D^{s}(f)$ the collection of all
$s-$gradients of $f$. The \textbf{homogeneous Haj{\l}asz-Sobolev
space} $\dot{M}^{s,X}(\Omega)$ consists of all
functions $f\in X$ for which%

\[
\left\|  f\right\|  _{\dot{M}^{s,X}\left(  \Omega\right)  }=\inf_{g\in
D^{s}(u)}\left\|  g\right\|  _{X}%
\]
is finite. The \textbf{Haj{\l}asz-Sobolev space} $M^{s,X}(\Omega)$ is $\dot{M}%
^{s,X}(S)\cap X$ equipped with the norm
\[
\left\|  f\right\|  _{M^{s,X}\left(  \Omega\right)  }=\left\|
f\right\|  _{X}+\left\|  f\right\|  _{\dot{M}^{s,X}\left(  \Omega\right)  }.
\]

When $X=L^{p}(\Omega),1\leq p\leq\infty,$ we shall write $M^{s,p}\left(
\Omega\right)  $ instead of $M^{s,X}\left(  \Omega\right)  .$

\begin{remark}
In the context of metric spaces, the spaces $M^{1,p}\left(
\Omega\right)  $ were first introduced by Haj{\l}asz (see \cite{Ha}
and \cite{Ha1}). They play an important role in the area of analysis
called analysis on metric spaces and a lot of papers have focused on
this subject (see for example \cite{Am}, \cite{HKT1}, \cite{Hei},
\cite{JK}, \cite{JKS},
 and the references quoted therein).\ When the measure
$\mu$ is doubling\footnote{$\mu$ is said to be doubling provided
there exists a constant $C>0$ such that
\par%
\[
\mu(2B)\leq C\mu(B)\text{ for all balls }B\subset\Omega.
\]
}, spaces $M^{1,X}(\Omega)$ have been considered in some particular
cases, for example, Haj{\l}asz-Lorentz-Sobolev spaces
$M^{1,L^{p,q}}\left(  \Omega\right)  $
(see \cite{JSYY}) and Musielak-Orlicz-Haj{\l}asz-Sobolev spaces $M^{1,L^{\Phi}%
}(\Omega),$ where $L^{\Phi}$ is an Orlicz space (see \cite{OO}).
Also in the doubling case, fractional spaces $M^{s,p}\left(
\Omega\right)  $ were introduced and studied in \cite{Yang} (see
also \cite{Hu} and \cite{HKT}).

For $p>1$, $M^{1,p}(\mathbb{R}^{n})=W^{1,p}\left(  \mathbb{R}^{n}\right)  $
(see \cite{Ha}), whereas for $p=1,$ $M^{1,1}(\mathbb{R}^{n})$ coincides with
the Hardy.Sobolev space $H^{1,1}(\mathbb{R}^{n})$ (see \cite[Thm 1]{KS}) and
if $0<s<1,$ then $M^{s,p}(\mathbb{R}^{n})=B_{p,\infty}^{s}(\mathbb{R}^{n})$
(see \cite{Yang}). Notice that in $\mathbb{R}^{n},$ if $s>1,$ then
$M^{s,p}(\mathbb{R}^{n})$ is trivial (contains only constant functions).
However, if $\Omega$ is a fractal, then $M^{s,p}(\mathbb{R}^{n})\ $for $s>1$
may be non-trivial (see \cite{Hu}).
\end{remark}

It is well known that the lower bound for the growth of the measure
\begin{equation}
\mu(B(x,r))\geq br^{\alpha},\label{bou1}%
\end{equation}
implies Sobolev embedding theorems for Haj{\l}asz-Sobolev spaces
$M^{1,p}$(see \cite{Ha} and \cite{Ha1}).

The converse problem, i.e. when the embedding
\begin{equation}
M^{1,p}(X)\subset L^{q}(X),\text{\ }q>p\label{emb}%
\end{equation}
implies a lower bound for the growth of the measure, has been
considered by several authors (see \cite{Go}, \cite{HKT}, \cite{KM},
\cite{Ka}, \cite{Ka1}, \cite{Ko} and the references quoted therein).
In the recent paper \cite{AGH}, R. Alvarado, P. G\'{o}rka and P.
Haj{\l}asz show that in fact if (\ref{emb}) holds with $q=\alpha
p/(\alpha-p)$, then lower bound for the growth (\ref{bou1}) holds.

The purpose of this paper is to obtain an analogous result for
$M^{s,X}$ spaces. This will be done by obtaining pointwise estimates
between the special difference $f^{\ast\ast}(t)-f^{\ast}(t)$ (called
the oscillation\footnote{Here $f^{\ast}$ is the decreasing
rearrangement of $f$, $f^{\ast\ast}(t)=\frac {1}{t}\int_0^t
f^{\ast}(s)ds$, for all $t>0$, (see Section \ref{Rea01}).} of $f$)
and the function $g$ (see Theorem \ref{the1} below), i.e. we will
see that for a wide range of measures, condition (\ref{bou1})
implies
\begin{equation}
f^{\ast\ast}(t)-f^{\ast}(t)\leq Ct^{s/\alpha}g^{\ast\ast}(t),\label{qeqd}%
\end{equation}
for every $f\in M^{s,L^{1}+L^{\infty}}$ and $g\in D^{s}(f).$ Moreover, if
$0<s\leq1,$ then (\ref{qeqd}) implies (\ref{bou1}).

Symmetrization inequalities imply Sobolev inequalities in the setting of
rearrangement invariant spaces. Indeed, from (\ref{qeqd}) we obtain: for any
r.i. space $X$ with upper Boyd\footnote{The restriction on the Boyd indices is
only required to guarantee that the inequality $\left\|  g^{\ast\ast}\right\|
_{X}\leq c_{X}\left\|  g\right\|  _{X},$ holds for all $g\in X.$} index
$\bar{\alpha}_{X}<1$, we have
\[
\left\|  t^{-s/\alpha}(f^{\ast\ast}(t)-f^{\ast}(t))\right\|  _{X}\leq
c\left\|  g\right\|  _{X},
\]
where $c=c(s,\alpha,X).$

Notice that we avoid one common drawback of the usual approaches to
Sobolev inequalities which require the choice of specific norms
before one starts the analysis. Instead, we work with pointwise
symmetrization inequalities which are *universal* and it is the
inequalities themselves that select the *correct* spaces.

For example, in the particular case of $X=L^{p}$ (see Corollary
\ref{coroLp} below) we obtain that if $1>s/\alpha>\frac{1}{p},$
then\footnote{As usual, the symbol $f\simeq g$ will indicate the
existence of a universal constant $c>0$ (independent of all
parameters involved) so that $(1/c)f\leq g\leq c\,f$, while the
symbol $f\preceq g$ means that $f\leq c\,g,$ and $f\succeq g$ means
that $f\geq c\,g.$}
\[
\left\|  t^{-s/\alpha}(f^{\ast\ast}(t)-f^{\ast}(t))\right\|  _{L^{p}}%
\simeq\left\|  t^{-s/\alpha}f^{\ast\ast}(t)\right\|  _{L^{p}}=\left\|
f\right\|  _{L^{p^{\ast},p}}%
\]
where $p_{s}^{\ast}=\frac{\alpha p}{\alpha-sp},$ i.e.
\[
\left(  \int_{0}^{\infty}\left(  t^{\frac{1}{p}-\frac{1}{\alpha}}f^{\ast\ast
}(t)\right)  ^{p}\frac{dt}{t}\right)  ^{1/p}\leq C\left(  \int_{\Omega}%
g^{p}\right)  ^{1/p}.
\]
On the other hand, since $p<p_{s}^{\ast}$ we have that
\[
L^{\frac{\alpha p}{\alpha-sp},p}\subset L^{\frac{\alpha p}{\alpha-sp}}%
\]
in particular, if $s=1$, then we get
\[
\left(  \int_{\Omega}\left|  f\right|  ^{p_{1}^{\ast}}d\mu\right)
^{1/p_{1}^{\ast}}\leq\left(  \int_{0}^{\infty}\left(  t^{\frac{1}{p}-\frac
{1}{\alpha}}f^{\ast\ast}(t)\right)  ^{p}\frac{dt}{t}\right)  ^{1/p}\leq
C\left(  \int_{\Omega}g^{p}d\mu\right)  ^{1/p}.
\]

\begin{remark}
The technique to obtain Sobolev oscillation type inequalities has
been developed by M. Milman and J. Mart\'{i}n (see \cite{MM3},
\cite{MM4} and \cite{MM6}) and provide a considerable simplification
in the theory of embeddings of Sobolev spaces based on rearrangement
invariant spaces.
\end{remark}

The paper is organized as follows. In Section 2, we introduce the
notation and the standard assumptions used in the paper, in Section
3, we will obtain oscillation type inequalities for spaces
$M^{s,X},$ we will see that they are equivalent to the lower bound
for the growth of the measure and will obtain Sobolev type embedding
of $M^{s,X}$ into a rearrangement invariant spaces. Finally, in the
appendix we will give some properties of the measures we will be
working with.

\section{Preliminaries}

In this section we establish some further notation and background
information and we provide more details about  metrics spaces and
r.i. spaces in will be working with.

\subsection{Metric spaces}

Let $(\Omega,d)$ be a metric space. As usual a ball $B$ in $\Omega$
with a center $x$ and radius $r>0$ is a set
$B=B(x,r):=\{y\in\Omega;d(x,y)<r\}$. Throughout the paper by a
metric measure space we mean a triple $(\Omega ,d,\mu)$, where $\mu$
is a Borel measure on $(\Omega,d)$ such $0<\mu (B)<\infty$, for
every ball $B$ in $\Omega$, we also assume that $\mu
(\Omega)=\infty$ and $\mu(\left\{ x\right\}  )=0$ for all
$x\in\Omega.$

We will say that a measure $\mu$ is $\alpha-$lower bounded if there are
$b,\alpha>0$ such that
\begin{equation}
\mu(B(x,r))\geq br^{\alpha},\label{boun}%
\end{equation}
For simplicity we assume in what follows that $\mu(B(x,r))\geq
r^{\alpha}.$

In what follows we will, furthermore, assume that the measure $\mu$
is continuous, i.e. $\mu$ satisfies that the map
$r\rightarrow\mu(B(x,r))$ is continuous\footnote{In the appendix we
describe measures with this property,} or that $\mu$ is doubling,
i.e. there exists a constant $C_{D}>1$ such that, for all
$x\in\Omega$ and for all $r>0,$ we have that
\begin{equation*}
\mu(B(x,2r))\leq C_{D}\mu(B(x,r))
\end{equation*}

Notice that in both cases there is a constant $c=c_{\mu}\geq1$ such
that given $t>0,$ for all $x\in\Omega,$ there is a positive number
$r(x)$ such that
\[
t\leq\mu(B(x,r(x))\leq ct.
\]
In the doubling case, given $x\in\Omega,$ consider
$r_{0}(x)=\sup\left\{ r:\mu(B(x,r))<t\right\}$ and take $r$  such
that $r<r_{0}(x)<2r,$ then
\[
t\leq\mu(B(x,2r))\leq C_{D}\mu(B(x,r))\leq C_{D}t.
\]
In what follows we call these measures $c-$almost
continuous\footnote{An example of an $\alpha-$lower bounded measure
that does not satisfy this condition is given in the appendix}.

\subsection{Background on Rearrangement Invariant Spaces\label{Rea01}}

For measurable functions $f:\Omega\rightarrow\mathbb{R},$ the distribution
function of $f$ is given by
\[
\mu_{f}(t)=\mu\{x\in{\Omega}:\left|  f(x)\right|  >t\}\text{ \ \ \ \ }(t>0).
\]
The \textbf{decreasing rearrangement} $f_{\mu}^{\ast}$ of $f$ is the
right-continuous non-increasing function from $[0,\infty)$ into $[0,\infty] $
which is equimeasurable with $f$. Namely,
\[
f_{\mu}^{\ast}(s)=\inf\{t\geq0:\mu_{f}(t)\leq s\}.
\]
We will write in what follows $f^{\ast}$ instead of $f_{\mu}^{\ast}.$

It is easy to see that for any measurable set $E\subset\Omega$
\begin{equation}
\int_{E}\left|  f(x)\right|  d\mu\leq\int_{0}^{\mu(E)}f^{\ast}%
(s)ds.\label{unoo}%
\end{equation}
Since $f^{\ast}$ is decreasing, the function $f^{\ast\ast},$ defined
by
\begin{equation}
f^{\ast\ast}(t)=\frac{1}{t}\int_{0}^{t}f^{\ast}(s)ds,\label{unoo1}
\end{equation}
is also decreasing and, moreover,
\[
f^{\ast}\leq f^{\ast\ast}.
\]

\begin{remark}
\label{oscil}An elementary computation shows that
\[
\frac{\partial}{\partial t}f^{\ast\ast}(t)=-\frac{f^{\ast\ast}(t)-f^{\ast}%
(t)}{t}%
\]
and that the function $t\rightarrow t\left(  f^{\ast\ast}(t)-f^{\ast
}(t)\right)  $ is increasing. Moreover, it is well known and easy to
see
\[
t\left(  f^{\ast\ast}(t)-f^{\ast}(t)\right)  =\int_{\{x\in{\Omega}:\left|
f(x)\right|  >f^{\ast}(t)\}}\left(  \left|  f(x)\right|  -f^{\ast}(t)\right)
dt.
\]
\end{remark}

\subsubsection{Rearrangement invariant spaces\label{rear}}

We recall briefly the basic definitions and conventions we use from
the theory of rearrangement-invariant (r.i.) spaces and refer the
reader to \cite{BS}, \cite{KPS}, for a complete treatment. We say
that a Banach function space $X=X({\Omega})$ on $({\Omega},d,\mu)$
is rearrangement-invariant (r.i.) space, if $g\in X$ implies that
all $\mu-$measurable functions $f$ with the same decreasing
rearrangement function with respect to the measure $\mu$, i.e. such
that $f^{\ast}=g^{\ast},$ also belong to $X,$ and, moreover, $\Vert
f\Vert _{X}=\Vert g\Vert_{X}$.

For any r.i. space $X({\Omega})$ we have%

\begin{equation*}
L^{\infty}(\Omega)\cap L^{1}(\Omega)\subset X(\Omega)\subset L^{1}%
(\Omega)+L^{\infty}(\Omega),
\end{equation*}
with continuous embedding.

A r.i. space $X({\Omega})$ can be represented by an r.i. space on
the interval $(0,\mu(\Omega)),$ with Lebesgue measure,
$\bar{X}=\bar{X}(0,\mu(\Omega)),$ such that
\[
\Vert f\Vert_{X}=\Vert f^{\ast}\Vert_{\bar{X}},
\]
for every $f\in X.$ A characterization of the norm $\Vert\cdot\Vert_{\bar{X}}$
is available (see \cite[Theorem 4.10 and subsequent remarks]{BS}). Typical
examples of r.i. spaces are the $L^{p}$-spaces, Lorentz spaces and Orlicz spaces.

The associated space $X^{\prime}(\Omega)$ of $X(\Omega)$ is the r.i. space of
all measurable functions $h$ for which the r.i. norm given by
\begin{equation}
\left\|  h\right\|  _{X^{\prime}(\Omega)}=\sup_{g\neq0}\frac{\int_{\Omega
}\left|  g(x)h(x)\right|  d\mu}{\left\|  g\right\|  _{X(\Omega)}%
}\label{holhol}%
\end{equation}
is finite. Note that by the definition (\ref{holhol}), the generalized
H\"{o}lder inequality%

\begin{equation}
\int_{\Omega}\left|  g(x)h(x)\right|  d\mu\leq\left\|  g\right\|  _{X(\Omega
)}\left\|  h\right\|  _{X^{\prime}(\Omega)}\label{holin}%
\end{equation}
holds.

The \textbf{fundamental function\ }of $X$ is defined by
\[
\phi_{X}(s)=\left\|  \chi_{E}\right\|  _{X},
\]
where $E$ is any measurable subset of $\Omega$ with $\mu(E)=s.$ We can assume
without loss of generality that $\phi_{X}$ is concave. Moreover,
\begin{equation}
\phi_{X^{\prime}}(s)\phi_{X}(s)=s.\label{dual}%
\end{equation}
Classically conditions on r.i. spaces are given in terms of the
Hardy defined by
\[
P^{(q)}f(t)=\left(  \dfrac{1}{t}\int_{0}^{t}\left|  f(x)\right|
^{q}\,\right)  ^{1/q},\text{ }Q_{\lambda}^{(q)}f(t)=\left(  \dfrac
{1}{t^{\lambda}}\int_{t}^{\infty}\left|  f(x)\right|  ^{q}\,\dfrac
{dx}{x^{1-\lambda}}\right)  ^{1/q},\text{ }t>0,
\]
here $0<q<\infty,\ 0\leq\lambda<1.$

The boundedness of these operators on r.i. spaces can be described in terms of
the so called \textbf{Boyd indices}\footnote{Introduced by D.W. Boyd in
\cite{boyd}.} defined by
\[
\bar{\alpha}_{X}=\inf\limits_{s>1}\dfrac{\ln h_{X}(s)}{\ln s}\text{ \ \ and
\ \ }\underline{\alpha}_{X}=\sup\limits_{s<1}\dfrac{\ln h_{X}(s)}{\ln s},
\]
where $h_{X}(s)$ denotes the norm of the compression/dilation operator $E_{s}
$ on $\bar{X}$, defined for $s>0,$ by $E_{s}f(t)=f^{\ast}(\frac{t}{s}) $. For
example if $X=L^{p}$ with $p>1$, then $\bar{\alpha}_{X}=\underline{\alpha}%
_{X}=\frac{1}{p}$. It is well known that (see \cite{MM1}, and \cite{MS})
\begin{equation}
\overline{\alpha}_{X}<\frac{1}{q}\Leftrightarrow P^{(q)}\text{ is bounded on
}X\text{,}\label{mm}%
\end{equation}%
\begin{equation}
\underline{\alpha}_{X}>\frac{\lambda}{q}\Leftrightarrow Q_{\lambda}%
^{(q)}\text{ is bounded on }X.\label{mmm}%
\end{equation}

\section{Symmetrization inequalities and embeddings for Fractional
Haj{\l}asz-Sobolev spaces\label{sime}}

The method of proof of the following Theorem follows the ideas of
\cite[Theorem 2]{MMPams} (see also \cite{Mas}).

\begin{theorem}
\label{the1}Let $\left(  \Omega,d,\mu\right)  $ be a metric measure
space such that $\mu$ is $c-$almost continuous and $\alpha-$lower
bounded. Let $f\in M^{s,L^{1}+L^{\infty}}\ $and $g\in D^{s}(f).$ Let
$0<p\leq1.$ Then, for all $t>0,$ we have
\begin{equation}
\left(  \left(  \left|  f\right|  ^{p}\right)  ^{\ast\ast}(t)-\left(  \left|
f\right|  ^{p}\right)  ^{\ast}(t)\right)  ^{1/p}\leq Ct^{s/\alpha}\left(
\left(  g^{p}\right)  ^{\ast\ast}(t)\right)  ^{1/p}\text{,}\label{jaja1}%
\end{equation}
where $C=C(c,p)$ is a constant that just depends on $c$ and $p.$
\end{theorem}

\begin{proof}
Let $0<p\leq1$, $f\in M^{s,L^{1}+L^{\infty}}(\Omega,d,\mu)$ and
$g\in D^{s}(f).$ Take $t>0$ and let
\[
A=\left\{  x\in\Omega:\left|  f(x)\right|  ^{p}>\left(  \left|  f\right|
^{p}\right)  ^{\ast}(t)\right\}  .
\]
Given $x\in A,$ since $\mu$ is $c-$almost continuous, there is a
radius $r(x) $
such that $\ $%
\[
2t\leq\mu(B(x,r(x)))\leq2ct.
\]
Let $r=\min((2t)^{1/\alpha},r(x)),$ and for every $x\in A,$ set
\[
A_{x}=\left\{  y\in B(x,r):\left|  f(y)\right|  ^{p}\leq\left(  \left|
f\right|  ^{p}\right)  ^{\ast}(t)\right\}  .
\]
From
\[
B(x,r)=\left(  B(x,r)\cap A\right)  \cup A_{x}%
\]
we see that
\begin{align*}
2t  & \leq\mu(B(x,r))\leq\mu(A)+\mu(A_{x})\\
& \leq t+\mu(A_{x}),
\end{align*}
whence
\begin{equation}
t\leq\mu(A_{x}).\label{dos}%
\end{equation}
Then
\begin{align*}
I  & =\int_{A}\left(  \left|  f(x)\right|  ^{p}-\left(  \left|  f\right|
^{p}\right)  ^{\ast}(t)\right)  d\mu(x)\\
& \leq\int_{A}\left(  \left|  f(x)\right|  ^{p}-\frac{1}{\mu\left(
A_{x}\right)  }\int_{A_{x}}\left|  f(y)\right|  ^{p}d\mu(y)\right)  d\mu(x)\\
& =\int_{A}\left(  \frac{1}{\mu\left(  A_{x}\right)  }\int_{A_{x}}\left(
\left|  f(x)\right|  ^{p}-\left|  f(y)\right|  ^{p}\right)  d\mu(y)\right)
d\mu(x)\\
& \leq\frac{1}{t}\int_{A}\int_{A_{x}}\left|  \left|  f(x)\right|  ^{p}-\left|
f(y)\right|  ^{p}\right|  d\mu(y)d\mu(x)\text{ \ (by (\ref{dos}))}\\
& \leq\frac{1}{t}\int_{A}\int_{B(x,r)}\left|  \left|  f(x)\right|
^{p}-\left|  f(y)\right|  ^{p}\right|  d\mu(y)d\mu(x)\text{ \ (since}%
A_{x}\subset B(x,r)\text{)}\\
& =J.
\end{align*}
Taking into account that
\[
\left|  \left|  x\right|  ^{p}-\left|  y\right|  ^{p}\right|  \leq\left|
x-y\right|  ^{p},
\]
we get
\begin{align*}
J  & \leq\frac{1}{t}\int_{A}\int_{B(x,r)}\left|  f(x)-f(y)\right|  ^{p}%
d\mu(y)d\mu(x)\\
& \leq\frac{1}{t}\int_{A}\int_{B(x,r)}d(x,y)^{sp}\left(  g(x)+g(y)\right)
^{p}d\mu(y)d\mu(x).\\
& \leq\frac{1}{t}\int_{A}\int_{B(x,r)}r^{sp}\left(  g(x)^{p}+g(y)^{p}\right)
d\mu(y)d\mu(x)\text{\ \ }\\
& \leq\frac{\left(  2t\right)
^{{sp/\alpha}}}{t}\int_{A}\int_{B(x,r)}\left(
g(x)^{p}+g(y)^{p}\right)  d\mu(y)d\mu(x)\\
& \leq\frac{\left(  2t\right)  ^{{sp/\alpha}}}{t}\left(  \int_{A}%
\int_{B(x,r)}g(x)^{p}d\mu(y)d\mu(x)+\int_{A}\int_{B(x,r)}g(y)^{p}d\mu
(y)d\mu(x)\right) \\
& \leq\frac{\left(  2t\right)  ^{{sp/\alpha}}}{t}\left(  \int_{A}g(x)^{p}%
\mu\left(  B(x,r)\right)  d\mu(x)+\int_{A}\left(  \int_{0}^{\mu\left(
B(x,r)\right)  }g^{\ast}(z)dz\right)  d\mu(x)\right) \text{ \ (by (\ref{unoo}))}\\
& \leq\frac{\left(  2t\right)  ^{{sp/\alpha}}}{t}\left(  ct\int_{0}%
^{t}g^{\ast}(z)^{p}dz+\int_{A}\left(  \int_{0}^{ct}g^{\ast}(z)^{p}dz\right)
d\mu(x)\right) \\
& \leq\frac{\left(  2t\right)  ^{{sp/\alpha}}}{t}\left(  ct\int_{0}%
^{t}g^{\ast}(z)^{p}dz+t\int_{0}^{ct}g^{\ast}(z)^{p}dz\right) \text{ \ (by (\ref{unoo1}))}\\
& \leq c2^{s/\alpha+1}t^{{sp/\alpha}}\int_{0}^{t}g^{\ast}(z)^{p}dz.
\end{align*}
Finally, the formula (see Remark \ref{oscil})
\[
t\left(  \left(  \left|  f\right|  ^{p}\right)  ^{\ast\ast}(t)-\left(  \left|
f\right|  ^{p}\right)  ^{\ast}(t)\right)  =\int_{A}\left(  \left|
f(x)\right|  ^{p}-\left(  \left|  f\right|  ^{p}\right)  ^{\ast}(t)\right)
d\mu(x)
\]
yields
\[
\left(  \left(  \left|  f\right|  ^{p}\right)  ^{\ast\ast}(t)-\left(
\left| f\right|  ^{p}\right)  ^{\ast}(t)\right)  ^{1/p}\leq
Ct^{s/\alpha}\left( \left(  g^{p}\right)  ^{\ast\ast}(t)\right)
^{1/p},
\]
as we wished to show.
\end{proof}

\begin{theorem}
Let $\left(  \Omega,d,\mu\right)  $ be a metric measure space such
that $\mu$ is $c- $almost continuous. Then the following statements
are equivalent

\begin{enumerate}
\item $\mu$ is $\alpha-$lower bounded.

\item  If $0<s\leq1,$ then for every $f\in M^{s,L^{1}+L^{\infty}}$ and $g\in
D^{s}(f)$, we have that
\begin{equation}
f^{\ast\ast}(t)-f^{\ast}(t)\leq Ct^{s/\alpha}g^{\ast\ast}(t)\text{.}%
\label{qeq}%
\end{equation}
\end{enumerate}
\end{theorem}

\begin{proof}
If $(1)$ holds, then by Theorem \ref{the1} we get $(2)$. Assume that
(\ref{qeq}) holds. Fix $x_{0}\in\Omega$ and $r>0$. Define the
function $f_{r},_{x_{0}}$ by
\[
f_{r},_{x_{0}}(x)=\left\{
\begin{array}
[c]{cc}%
\left(  r-d(x_{0},x)\right)  ^{s} & \text{ if }d(x_{0},x)\leq r,\\
0 & \text{if }d(x_{0},x)>r
\end{array}
\right.
\]
It is easy seem that $g_{r},_{x_{0}}(x)=\chi_{B(x_{0},r)}$, satisfies that
\[
\left|  f_{r},_{x_{0}}(x)-f_{r},_{x_{0}}(y)\right|  \leq d(x,y)^{s}\left(
g_{r},_{x_{0}}(x)+g_{r},_{x_{0}}(y)\right)  .
\]
An elementary computation shows that
\begin{subequations}
\begin{equation}
\mu_{f_{r},_{x_{0}}}(\lambda)=\left\{
\begin{array}
[c]{cc}%
\mu\left(  B(x_{0},r-\lambda^{1/s}\right)  & \text{if }0<\lambda\leq r^{s},\\
0 & \text{if }\lambda>r^{s}.
\end{array}
\right. \label{dist}%
\end{equation}
By hypothesis,
\end{subequations}
\[
\left(  f_{r},_{x_{0}}\right)  ^{\ast\ast}(t)-\left(  f_{r},_{x_{0}}\right)
^{\ast}(t)\leq Ct^{s/\alpha}\left(  g_{r},_{x_{0}}\right)  ^{\ast\ast}(t)
\]
Thus,
\begin{align*}
\left(  f_{r},_{x_{0}}\right)  ^{\ast\ast}(0)-\left(  f_{r},_{x_{0}}\right)
^{\ast\ast}(2\mu(B(x_{0},r)))  & =\int_{0}^{2\mu(B(x_{0},r))}\left(  \left(
f_{r},_{x_{0}}\right)  ^{\ast\ast}(t)-\left(  f_{r},_{x_{0}}\right)  ^{\ast
}(t)\right)  \frac{dt}{t}\\
& \leq C\int_{0}^{2\mu(B(x_{0},r))}t^{s/\alpha-1}\left(  g_{r},_{x_{0}%
}\right)  ^{\ast\ast}(t)dt.
\end{align*}
But
\[
\left(  g_{r},_{x_{0}}\right)  ^{\ast\ast}(t)=\frac{1}{t}\int_{0}^{t}%
\chi_{\lbrack0,\mu(B(x_{0},r))}(s)ds=\min\left(  1,\frac{\mu(B(x_{0},r))}%
{t}\right)  ,
\]
thus
\begin{align*}
\int_{0}^{2\mu(B(x_{0},r))}t^{s/\alpha-1}\left(  g_{r},_{x_{0}}\right)
^{\ast\ast}(t)dt  & =\int_{0}^{\mu(B(x_{0},r))}t^{s/\alpha-1}dt+\mu
(B(x_{0},r))\int_{\mu(B(x_{0},r))}^{2\mu(B(x_{0},r))}t^{s/\alpha-2}dt\\
& \preceq\mu(B(x_{0},r))^{s/\alpha}.
\end{align*}
On the other hand
\begin{equation}
\left(  f_{r},_{x_{0}}\right)  ^{\ast\ast}(0)=\left\|  f_{r},_{x_{0}}\right\|
_{L^{\infty}}=r^{s}\label{Linf}%
\end{equation}
so from (\ref{dist}) we get that $\left(  f_{r},_{x_{0}}\right)
^{\ast}(t)=0 $ if $t>\mu(B(x_{0},r),$ therefore
\begin{align}
\left(  f_{r},_{x_{0}}\right)  ^{\ast\ast}(2\mu(B(x_{0},r)))  & =\frac{1}%
{2\mu(B(x_{0},r)}\int_{0}^{2\mu(B(x_{0},r)}\left(  f_{r},_{x_{0}}\right)
^{\ast}(t)dt\label{L1}\\
& =\frac{1}{2\mu(B(x_{0},r)}\int_{0}^{\mu(B(x_{0},r)}\left(  f_{r},_{x_{0}%
}\right)  ^{\ast}(t)dt\nonumber\\
& =\frac{1}{2\mu(B(x_{0},r)}\left\|  f_{r},_{x_{0}}\right\|  _{L^{1}%
}\nonumber\\
& \leq\frac{1}{2\mu(B(x_{0},r)}\left(  r^{s}\mu(B(x_{0},r)\right)
=\frac{r^{s}}{2}.\nonumber
\end{align}
Combining (\ref{Linf}) and (\ref{L1}) we obtain
\[
\frac{r^{s}}{2}\preceq\mu(B(x_{0},r))^{s/\alpha}.
\]
which implies that $\mu$ is $\alpha-$lower bounded up to constants.
\end{proof}

Theorem \ref{sime} provides us the following Sobolev embedding
result for Fractional Haj{\l}asz-Sobolev spaces.

\begin{theorem}
Let $\left(  \Omega,d,\mu\right)  $ be a metric measure space such
that $\mu$ is $c- $almost continuous and $\alpha-$lower bounded. Let
$X$ be a r.i. space. Let $f\in M^{s,X}$ and $g\in D^{s}(f)$
\begin{enumerate}
\item  If $s/\alpha<1$,

\begin{enumerate}
\item  If $\underline{\alpha}_{X}>s/\alpha,$ then
\[
\left\|  t^{-s/\alpha}f^{\ast\ast}(t)\right\|  _{\bar{X}}\preceq\left\|
g\right\|  _{X}.
\]

\item  If $\bar{\alpha}_{X}<s/\alpha,$ then
\[
\left\|  f\right\|  _{L^{\infty}}\leq\left\|  g\right\|  _{X}+\left\|
f\right\|  _{L^{1}+L^{\infty}}.
\]
\
\end{enumerate}

\item  If $s/\alpha=1,$ then
\[
\sup_{t>0}\phi_{X}(t)\frac{\left(  f^{\ast\ast}(t)-f^{\ast}(t)\right)  }%
{t}\preceq\left\|  g\right\|  _{X}.
\]

\item If $s/\alpha>1,$ then
\[
\left\|  f\right\|  _{L^{\infty}}\leq\left\|  g\right\|  _{X}+\left\|
f\right\|  _{L^{1}+L^{\infty}}%
\]
\end{enumerate}
\end{theorem}

\begin{proof}

Case (1) Assume $s/\alpha<1,$ then:

\

(a)  Condition $\underline{\alpha}_{X}>s/\alpha,$ implies
$f^{\ast\ast}(\infty)=0,$ so
\[
t^{-s/\alpha}f^{\ast\ast}(t)=t^{-s/\alpha}\int_{t}^{\infty}z^{\frac{s}{\alpha
}}\frac{\left(  f^{\ast\ast}(z)-f^{\ast}(z)\right)  }{z^{\frac{s}{\alpha}}%
}\frac{dz}{z}=Q_{\frac{s}{\alpha}}\left[  \frac{\left(  f^{\ast\ast}%
(\cdot)-f^{\ast}(\cdot)\right)  }{\left(  \cdot\right)  ^{\frac{s}{\alpha}}%
}\right]  (t).
\]
Hence
\begin{align*}
\left\|  t^{-s/\alpha}f^{\ast\ast}(t)\right\|  _{\bar{X}} &
=\left\|
Q_{\frac{s}{\alpha}}\left[  \frac{\left(  f^{\ast\ast}(\cdot)-f^{\ast}%
(\cdot)\right)  }{\left(  \cdot\right)  ^{\frac{s}{\alpha}}}\right]
(t)\right\|  _{\bar{X}}\\
&  \leq C\left\|  \left(  \frac{f^{\ast\ast}(t)-f^{\ast}(t)}{t^{s/\alpha}%
}\right)  \right\|  _{\bar{X}}\text{ (since }\underline{\alpha}_{X}%
>s/\alpha).\\
&  \leq C\left\|  t^{-s/\alpha}f^{\ast\ast}(t)\right\|  _{\bar{X}}.
\end{align*}
Therefore, if $\bar{\alpha}_{X}<1,$ we have that
\begin{align*}
\left\|  t^{-s/\alpha}f^{\ast\ast}(t)\right\|  _{\bar{X}} &
\simeq\left\| \left(
\frac{f^{\ast\ast}(t)-f^{\ast}(t)}{t^{s/\alpha}}\right)  \right\|
_{\bar{X}}\\
&  \preceq\left\|  g^{\ast\ast}(t)\right\|  _{\bar{X}}\\
&  \preceq\left\|  g\right\|  _{X}.
\end{align*}

In case that $\bar{\alpha}_{X}=1,$ let $0<p<1,$ and consider the
function $\left|  f\right|  ^{p}.$ By (\ref{jaja1}) we have that
\begin{equation}
\left(  t^{-sp/\alpha}\left(  \left|  f\right|  ^{p}\right)
^{\ast\ast
}(t)-\left(  \left|  f\right|  ^{p}\right)  ^{\ast}(t)\right)  ^{1/p}%
\preceq\left(  \left(  g^{p}\right)  ^{\ast\ast}(t)\right)  ^{1/p}.
\label{zaza}%
\end{equation}
The formula
\begin{align*}
\left(  \left|  f\right|  ^{p}\right)  ^{\ast\ast}(t)  &
=\int_{t}^{\infty }\left(  \left(  \left|  f\right|  ^{p}\right)
^{\ast\ast}(z)-\left(  \left|
f\right|  ^{p}\right)  ^{\ast}(z)\right)  \frac{dz}{z}\\
&  =\int_{t}^{\infty}z^{\frac{sp}{\alpha}}\left(  \left[  z^{-\frac{sp}%
{\alpha}}\left(  \left(  \left|  f\right|  ^{p}\right)
^{\ast\ast}(z)-\left( \left|  f\right|  ^{p}\right)
^{\ast}(z)\right)  \right]  ^{1/p}\right) ^{p}\frac{dz}{z},
\end{align*}
yields
\begin{align*}
t^{-\frac{s}{\alpha}}\left(  \left(  \left|  f\right|  ^{p}\right)
^{\ast \ast}(t)\right)  ^{1/p}  &  =\left(
t^{-\frac{sp}{\alpha}}\int_{t}^{\infty }z^{\frac{sp}{\alpha}}\left(
z^{-\frac{s}{p\alpha}}\left[  \left(  \left( \left|  f\right|
^{p}\right)  ^{\ast\ast}(z)-\left(  \left|  f\right|
^{p}\right)  ^{\ast}(z)\right)  \right]  ^{1/p}\right)  ^{p}\frac{dz}%
{z}\right)  ^{1/p}\\
&  =Q_{\frac{sp}{\alpha}}^{\left(  p\right)  }\left(  \left(
\cdot\right) ^{\frac{-s}{p\alpha}}\left[  \left(  \left(  \left|
f\right|  ^{p}\right)
^{\ast\ast}(\cdot)-\left(  \left|  f\right|  ^{p}\right)  ^{\ast}%
(\cdot)\right)  \right]  ^{1/p}\right)  (t),
\end{align*}
Since $\underline{\alpha}_{X}>\frac{s}{\alpha},$ the operator
$Q_{\frac {s}{\alpha}}^{\left(  p\right)  }$ is bounded on $X$ (by
\text{(\ref{mmm}))}, thus
\begin{align*}
\left\|  t^{-\frac{s}{\alpha}}f^{\ast\ast}\right\|  _{\bar{X}}  &
\simeq\left\|  t^{-\frac{s}{\alpha}}\left(  \left(  \left|  f\right|
^{p}\right)  ^{\ast\ast}(t)\right)  ^{1/p}\right\|  _{\bar{X}}\\
&  =\left\|  Q_{\frac{s}{p\alpha}}^{\left(  p\right)  }\left( \left(
\cdot\right)  ^{\frac{-s}{p\alpha}}\left[  \left(  \left( \left|
f\right| ^{p}\right)  ^{\ast\ast}(\cdot)-\left(  \left| f\right|
^{p}\right)  ^{\ast
}(\cdot)\right)  \right]  ^{1/p}\right)  (t)\right\|  _{\bar{X}}\\
&  \preceq\left\|  \left(  z^{-\frac{s}{p\alpha}}\left[  \left(
\left( \left|  f\right|  ^{p}\right)  ^{\ast\ast}(z)-\left(  \left|
f\right|
^{p}\right)  ^{\ast}(z)\right)  \right]  ^{1/p}\right)  \right\|  _{\bar{X}}\\
&  \preceq\left\|  \left(  \left(  g^{p}\right)
^{\ast\ast}(t)\right)
^{1/p}\right\|  _{\bar{X}}\text{ \ (by (\ref{zaza}))}\\
&  \preceq\left\|  g\right\|  _{X}\text{ \ (by (\ref{mm}))}.
\end{align*}
\qquad

(b) If $\bar{\alpha}_{X}<s/\alpha,$  by Theorem 2.3 of \cite{FMP},
we have that
\begin{align*}
\left\|  f\right\|  _{L^{\infty}}  &  \leq\left\|  \frac{f^{\ast\ast
}(t)-f^{\ast}(t)}{t^{s/\alpha}}\right\|  _{\bar{X}}+\left\|
f\right\|
_{L^{1}+L^{\infty}}\\
&  \preceq\left\|  g^{\ast\ast}(t)\right\|  _{\bar{X}}\preceq\left\|
g\right\|  _{X}.
\end{align*}

Case (2)  If $s/\alpha=1,$ then%

\[
\frac{f^{\ast\ast}(t)-f^{\ast}(t)}{t}\leq
Cg^{\ast\ast}(t)=\frac{1}{t}\int
_{0}^{t}g^{\ast}(s)ds\leq\left\|  g\right\|  _{X}\frac{\phi_{X^{\prime}}%
(t)}{t}\text{ (by (\ref{holin}))}.
\]
Consequently,
\[
\phi_{X}(t)\frac{\left(  f^{\ast\ast}(t)-f^{\ast}(t)\right)  }{t}%
\preceq\left\|  g\right\|  _{X} \text{ (by (\ref{dual}))}.
\]

Case (3) If $s/\alpha>1,$ then
\[
\frac{f^{\ast\ast}(t)-f^{\ast}(t)}{t^{s/\alpha}}\preceq\frac{1}{t}\int_{0}%
^{t}g^{\ast}(s)ds\leq\left\|  g\right\|
_{X}\frac{\phi_{X^{\prime}}(t)}{t}.
\]
On the other hand
\begin{align*}
f^{\ast\ast}(0)-f^{\ast\ast}(1)  &  =\int_{0}^{1}\left(  f^{\ast\ast
}(z)-f^{\ast}(z)\right)
\frac{dz}{z}=\int_{0}^{1}z^{s/\alpha-1}\left(
\frac{f^{\ast\ast}(z)-f^{\ast}(z)}{z^{s/\alpha}}\right)  dz\\
&  \preceq\left\|  g\right\|  _{X}\int_{0}^{1}z^{s/\alpha-1}\frac
{\phi_{X^{\prime}}(z)}{z}dz\\
&  \leq\left\|  g\right\|
_{X}\phi_{X^{\prime}}(1)\int_{0}^{1}z^{s/\alpha
-2}dz\\
&  \preceq\left\|  g\right\|  _{X}.
\end{align*}
Finally, using that $\left\|  f\right\|
_{L^{\infty}}=f^{\ast\ast}(0)$, we obtain
\[
\left\|  f\right\|  _{L^{\infty}}\leq\left\|  g\right\|
_{X}+\left\| f\right\|  _{L^{1}+L^{\infty}}.
\]
As we wished to show
\end{proof}

\begin{remark}
In the classical setting, embedding theorems for $W^{1,}{}^{p}(\mathbb{R}%
^{n})$ have different behavior when $p<n,$ $p=n,$or $p>n$. The point is that
\begin{equation}
\left|  B(x,r)\right|  \simeq r^{n},\label{jaja}.
\end{equation}
In the metric-measure context the counterpart condition (\ref{jaja})
is provided by the lower bound for the growth of the measure
(\ref{boun}). Notice also the different character on the embeddings
for $p<n,$ $p=n,$ or $p>n$ in the r.i. context is done by the role
of the Boyd index.
\end{remark}

\begin{corollary}
\label{coroLp}Under conditions of the previous Theorem, in the particular case
of $X=L^{p}$, we obtain

\begin{enumerate}
\item  If $s/\alpha<1$,

\begin{enumerate}
\item  If $s/\alpha>\frac{1}{p},$ then
\[
\left\|  f\right\|  _{L^{p^{\ast},p}}\preceq\left\|  g\right\|  _{L^{p}}.
\]
where $p^{\ast}=\frac{\alpha p}{\alpha-sp}.$

\item  If $\frac{1}{p}=s/\alpha,$ then
\[
\left(  \int_{0}^{1}\left(  \frac{f^{\ast\ast}(t)}{1+\ln\left(  \frac
{1}{t}\right)  }\right)  ^{p}\frac{dt}{t}\right)  ^{1/p}\preceq\left\|
g\right\|  _{L^{p}}+\left\|  f\right\|  _{L^{1}+L^{\infty}}.
\]

\item  If $\frac{1}{p}<s/\alpha,$ then
\[
\left\|  f\right\|  _{L^{\infty}}\leq\left\|  g\right\|  _{L^{p}}+\left\|
f\right\|  _{L^{1}+L^{\infty}}.
\]
\
\end{enumerate}

\item  If $s/\alpha=1,$ then
\[
\sup_{t>0}t^{1/p}\frac{\left(  f^{\ast\ast}(t)-f^{\ast}(t)\right)  }{t}%
\preceq\left\|  g\right\|  _{L^{p}}.
\]
In particular, if $L^{p}=L^{1},$ we obtain $\left\|  f\right\|  _{L_{w}%
^{\infty}}\preceq\left\|  f\right\|  _{M^{s,L^{p}}},$ where $L_{w}^{\infty}$
is the weak $L^{\infty}-$space\footnote{Bennett, DeVore and Sharpley
introduced the space weak $L^{\infty}$ defined as $\left\|  f\right\|
_{L_{w}^{\infty}}=\sup_{t>0}\left(  f^{\ast\ast}(t)-f^{\ast}(t)\right)  $ in
\cite{BDS} where studied its relationship with functions of bounded mean oscillation}

\item  If $s/\alpha>1,$ then
\[
\left\|  f\right\|  _{L^{\infty}}\leq\left\|  g\right\|  _{L^{p}}+\left\|
f\right\|  _{L^{1}+L^{\infty}}.
\]
\end{enumerate}
\end{corollary}
\begin{proof}
Except $1.b$, the remaining statements are a particular case of the
previous Theorem.
To see 1.b, by Theorem \ref{the1}, we have that%
\[
t^{-1/p}\left(  f^{\ast\ast}(t)-f^{\ast}(t)\right)  \leq
Cg^{\ast\ast}(t).
\]
Hence%
\[
\left(  \int_{0}^{\infty}\left(  t^{-1/p}\left(
f^{\ast\ast}(t)-f^{\ast
}(t)\right)  \right)  ^{p}dt\right)  ^{1/p}\preceq\left\|  g\right\|  _{L^{p}%
},
\]
and by \cite[Lemma 5.4]{BMF} we have that
\[
\left(  \int_{0}^{1}\left(  \frac{f^{\ast\ast}(t)}{1+\ln\left( \frac
{1}{t}\right)  }\right)  ^{p}\frac{dt}{t}\right) ^{1/p}\preceq\left(
\int_{0}^{\infty}\left(  t^{-1/p}\left(
f^{\ast\ast}(t)-f^{\ast}(t)\right) \right)  ^{p}dt\right)
^{1/p}+\left\|  f\right\|  _{L^{1}+L^{\infty}}.
\]
\end{proof}

\section{Appendix}

\subsection{Description of measures $\mu$ satisfying that the map
$r\rightarrow\mu(B(x,r))$ is continuous}

Let $(\Omega,d)$ be a metric measure space, we say that a measure
$\mu$ is metrically continuous with respect to metric $d$ if all
$x\in\Omega$ and all $r>0$ it holds that
\[
\lim_{d(x,y)\rightarrow0}\mu(B(x,r)\Delta B(y,r))=0
\]
where $A\Delta B$ stands for a symmetric difference of sets
$A,B\subset\Omega$ and is defined as follows: $A\Delta
B:=(A\backslash B)\cup(B\backslash A).$

The following lemma collects some basic facts about continuity of a
measure with respect to the metric (see \cite{GG} and \cite{Ad} for
the proof).

\begin{lemma}
Let $(\Omega,d,%
\mu
)$ be a metric space with a Borel regular measure $\mu$. Then the
following hold:

\begin{enumerate}
\item  If $\mu$ is continuous with respect to the metric $d$, then the map
$x\rightarrow\mu(B(x,r))$ is continuous in $d$.

\item  If for every $x\in\Omega$ and every $r>0$ it holds that $\mu(\partial
B(x,r))=0$, then $\mu$ is continuous with respect to the metric $d$.

\item  If for every $x\in\Omega$ the function $r\rightarrow\mu(B(x,r))$ is
continuous, then $\mu$ is continuous with respect to the metric $d$.
\end{enumerate}
\end{lemma}

It is easy to see that if we take $\mathbb{R}^{n}$ with Lebesgue
measure (or with an absolutely continuous measure respect to the
Lebesgue measure) with the Euclidean distance, then this measure is
metrically continuous. In fact we
have more (see \cite[Proposition 2.1]{GG} ) if $(\Omega,d,%
\mu
)$ and $(\Omega,d,\nu)$ are metric measure spaces then if
$\mu\prec\prec\nu$ and $\nu$ is metrically continuous, then $\mu$ is
metrically continuous too.

An important example is the following (see \cite{MMNO}).

\begin{lemma}
Let $%
\mu
$ be a nonnegative Radon measure on $\mathbb{R}^{n}$. Assume that
for any
point $p\in\mathbb{R}^{n}$, $%
\mu
(\{p\})=0$, \ then we choose the coordinate axes in such a way that $%
\mu
(\partial Q)=0$ for all cubes $Q$ with sides parallel to the axes.
In particular the function $\ell\rightarrow\mu(Q(x,\ell))$ where
$\ell$ denotes the length of the edge and $x$ is the center of $Q$,
is continuous.
\end{lemma}

\subsection{An example of an $\alpha-$lower bounded measure which is not
$c-$almost continuous}

Consider $\mathbb{R}^{2}$ with the distance
$d_{\infty}(x,y)=\max\left\{ \left|  x\right|  ,\left|  y\right|
\right\}  $ and the measure  $\mu=$ Lebesgue measure in the plane +
length measure in vertical axis + length measure in vertical
straight line passing through $(1,0)$. It is easy to see that
\[
\frac{r^{2}}{4}\leq\mu(B(x,r))
\]
however, is not $c-$almost continuous.

\

{\bf Acknowledgements} The authors are grateful to professor Xavier
Tolsa for providing us the above example.

\end{document}